\documentclass[12pt]{amsart}

\usepackage{amssymb}
\usepackage{amsthm}
\usepackage{amsmath}
\usepackage{verbatim}
\usepackage{mathrsfs}
\usepackage{mdwlist}
\usepackage{a4wide}

\theoremstyle{theorem}
\newtheorem*{theorem}{Theorem}
\newtheorem{lemma}{Lemma}

\theoremstyle{definition}
\newtheorem{remark}{Remark}

\newcommand{\R}{\mathbb{R}}

\newcommand{\N}{\mathbb{N}}

\newcommand{\p}{\varphi}

\newcommand{\w}{\widetilde}
\newcommand{\oo}{\overline}

\newcommand{\Om}{\Omega}
\newcommand{\I}{\mathscr{I}}
\newcommand{\Or}{\mathrm{O}}

\begin{document}
\title[Almost orthogonally additive functions]{Almost orthogonally additive functions}

\author[T. Kochanek]{Tomasz Kochanek}
\address{Institute of Mathematics\\ University of Silesia\\
Bankowa 14\\ 40-007 Katowice, Poland}
\email{tkochanek@math.us.edu.pl}
\author[W. Wyrobek-Kochanek]{Wirginia Wyrobek-Kochanek}
\address{Institute of Mathematics\\ University of Silesia\\
Bankowa 14\\ 40-007 Katowice, Poland}
\email{wwyrobek@math.us.edu.pl}

\thanks{This research has been supported by the scholarship from the UPGOW project co-financed by the European Social Fund}

\begin{abstract}
If a function $f$, acting on a Euclidean space $\R^n$, is \lq\lq almost\rq\rq\ orthogonally additive in the sense that $f(x+y)=f(x)+f(y)$ for all $(x,y)\in\bot\setminus Z$, where $Z$ is a \lq\lq negligible\rq\rq\ subset of the $(2n-1)$-dimensional manifold $\bot\subset\R^{2n}$, then $f$ coincides almost everywhere with some orthogonally additive mapping.
\end{abstract}

\maketitle

\section{Introduction}Let $(E,\langle\cdot\vert\cdot\rangle)$ be a real inner product space, $\dim E\geq 2$, and let $(G,+)$ be an Abelian group. A~function $f\colon E\to G$ is called {\it orthogonally additive\/} iff it satisfies the equation
\begin{equation}\label{OA}
f(x+y)=f(x)+f(y)
\end{equation}
for all $(x,y)\in\bot:=\{(x,y)\in E^2: \langle x\vert y\rangle=0\}$. It was proved independently by R.~Ger, Gy.~Szab\'o and J.~R\"atz \cite[Corollary 10]{ratz} that such a function has the form
\begin{equation}\label{1}
f(x)=a\left(\| x\|^2\right)+b(x)
\end{equation}
with some additive mappings $a\colon\R\to G$, $b\colon E\to G$ provided that $G$ is uniquely $2$-divisible. This divisibility assumption was dropped by K.~Baron and J.~R\"atz \cite[Theorem 1]{baron_ratz}.

We are going to deal with the situation where equality \eqref{OA} holds true for all orthogonal pairs $(x,y)$ outside from a \lq\lq negligible\rq\rq\ subset of $\bot$. Considerations of this type go back to a problem \cite{erdos}, posed by P. Erd\H{o}s, concerning the unconditional version of Cauchy's functional equation \eqref{OA}. It was solved by N.~G.~de Bruijn \cite{bruijn} and, independently, by W.~B.~Jurkat \cite{jurkat}, and also generalized by R.~Ger \cite{ger2}. Similar research concerning mappings which preserve inner product was made by J.~Chmieli\'nski and J.~R\"atz \cite{chmielinski_ratz} and by J.~Chmieli\'nski and R.~Ger \cite{chmielinski_ger}.

While studying unconditional functional equations, \lq\lq negligible\rq\rq\ sets are usually understood as the members of some proper linearly invariant ideal. Moreover, any such ideal of subsets of an underlying space $X$ automatically generates another such ideal of subsets of $X^2$ via the Fubini theorem (see R. Ger \cite{ger1} and M. Kuczma \cite[\S 17.5]{kuczma}). However, we shall assume that equation \eqref{OA} is valid for $(x,y)\in\bot\setminus Z$, where $Z$ is \lq\lq negligible\rq\rq\ in $\bot$ (not only in $E^2$), and therefore the structure of  $\bot$ should be appropriate to work with \lq\lq linear invariance\rq\rq\ and Fubini-type theorems. This is the reason why we restrict our attention to Euclidean spaces $\R^n$ and regard $\bot$ as a smooth $(2n-1)$-dimensional manifold lying in $\R^{2n}$.

\section{Preliminary results}For completeness let us recall some definitions concerning the manifold theory (for further information see, e.g., R. Abraham, J. E. Marsden and T. Ratiu \cite{abraham}, and L. W. Tu \cite{tu}). Let $S$ be a topological space; by an {\it $m$-dimensional $\mathcal{C}^\infty$-atlas\/} we mean a family $\mathcal{A}=\{(U_i,\p_i)\}_{i\in I}$ such that $\{U_i\}_{i\in I}$ is an open covering of $S$, for each $i\in I$ the mapping $\p_i$ is a homeomorphism which maps $U_i$ onto an open subset of $\R^m$, and for each $i,j\in I$ the mapping $\p_i\circ\p_j^{-1}$ is a $\mathcal{C}^\infty$-diffeomorphism defined on $\p_j(U_i\cap U_j)$. Brouwer's theorem of dimension invariance implies that each two atlases on $S$ are of the same dimension.

We say that atlases $\mathcal{A}_1$ and $\mathcal{A}_2$ are {\it equivalent\/} iff $\mathcal{A}_1\cup\mathcal{A}_2$ is an atlas. A {\it $\mathcal{C}^\infty$-differentiable structure\/} $\mathcal{D}$ on $S$ is an equivalence class of atlases on $S$; the union $\bigcup\mathcal{D}$ forms a maximal atlas on $S$ and any of its element is called an {\it admissible chart\/}. By a {\it $\mathcal{C}^\infty$-differentiable manifold\/} (briefly: {\it manifold\/}) $M$ we mean a pair $(S,\mathcal{D})$ of a topological space $S$ and a $\mathcal{C}^\infty$-differentiable structure $\mathcal{D}$ on $S$; we shall then identify $M$ with the space $S$ for convenience. A manifold is called an {\it $m$-manifold\/} iff its every atlas is $m$-dimensional.

Having an $m_1$-manifold $M_1=(S_1,\mathcal{D}_1)$ and an $m_2$-manifold $M_2=(S_2,\mathcal{D}_2)$ we may define the {\it product manifold\/} $M_1\times M_2=(S_1\times S_2,\mathcal{D}_1\times\mathcal{D}_2)$, where the differentiable structure $\mathcal{D}_1\times\mathcal{D}_2$ is generated by the atlas $$\left\{(U_1\times U_2,\p_1\times\p_2):\, (U_i,\p_i)\in\bigcup\mathcal{D}_i\mbox{ for }i=1,2\right\}.$$Then $M_1\times M_2$ forms an $(m_1+m_2)$-manifold. For an arbitrary set $A\subset M_1\times M_2$ and any point $x\in M_1$ we use the notation $A[x]=\{y\in M_2:\, (x,y)\in A\}$.

In what follows, we consider only manifolds $M\subset\R^n$, for some $n\in\N$, equipped with the natural topology and a differentiable structure which is determined by the following condition: for every $x\in M$ there is a $\mathcal{C}^\infty$-diffeomorphism $\p$ defined on an open set $U\subset\R^n$ with $x\in U$ such that $\p(M\cap U)=\p(U)\cap (\R^m\times\{ 0\})$, where $m$ is the dimension of $M$. In particular, every open subset of $\R^n$ yields an $n$-manifold with the atlas consisting of a single identity map. Any set $M\subset\R^n$ satisfying the above condition forms a {\it submanifold\/} of $\R^n$ in the sense of \cite[Definition 3.2.1]{abraham}, or a {\it regular submanifold\/} of $\R^n$ in the sense of \cite[Definition 9.1]{tu}. Generally, if $M_1$ is an $m_1$-manifold and $M_2$ is an $m_2$-manifold, then $M_1$ is called a ({\it regular\/}) {\it submanifold\/} of $M_2$ iff $M_1\subset M_2$ and for every $x\in M_1$ there is an admissible chart $(U,\p)$ of $M_2$ with $x\in U$ such that $\p(M_1\cap U)=\p(U)\cap (\R^{m_1}\times\{0\})$.

If $M_1$ and $M_2$ are manifolds with atlases $\mathcal{A}_1$ and $\mathcal{A}_2$, respectively, then a mapping $\Phi\colon M_1\to M_2$ is said to be {\it of the class $\mathcal{C}^\infty$\/} iff it is continuous and for all $(U,\p)\in\mathcal{A}_1$, $(V,\psi)\in\mathcal{A}_2$ the composition $\psi\circ\Phi\circ\p^{-1}$ is of the class $\mathcal{C}^\infty$ (in the usual sense) in its domain. This condition is independent on the choice of particular atlases generating differentiable structures of $M_1$ and $M_2$; see \cite[Proposition 3.2.6]{abraham}. We say that $\Phi$ is a {\it $\mathcal{C}^\infty$-diffeomorphism\/} iff $\Phi$ is a bijection between $M_1$ and $M_2$, and both $\Phi$ and $\Phi^{-1}$ are of the class $\mathcal{C}^\infty$. According to the above explanation, such a definition is compatible with the usual notion of a $\mathcal{C}^\infty$-diffeomorphism. If any $\mathcal{C}^\infty$-diffeomorphism between $M_1$ and $M_2$ exists, then we write $M_1\sim M_2$. Of course, in such a case the manifolds $M_1$ and $M_2$ are of the same dimension.

Finally, a mapping $\Phi\colon M_1\to M_2$ between an $m_1$-manifold $M_1$ and an $m_2$-manifold $M_2$ is called a $\mathcal{C}^\infty$-{\it immersion\/} [$\mathcal{C}^\infty$-{\it submersion\/}] iff it is of the class $\mathcal{C}^\infty$ and for every $x\in M_1$ there exist admissible charts $(U,\p)$ and $(V,\psi)$ of $M_1$ and $M_2$, respectively, such that $x\in U$, $\Phi(x)\in V$, and the derivative of the function $\psi\circ\Phi\circ\p^{-1}$ at any point of $\p(U)$ is an injective [a surjective] linear mapping from $\R^{m_1}$ to $\R^{m_2}$ (see \cite[Proposition 8.12]{tu} for another, equivalent definition). We will find the following lemma useful; for the proof see R. W. R. Darling \cite[\S 5.5.1]{darling}.
\begin{lemma}\label{submersion}
Let $M_1$ be a submanifold of an open set $U\subset\R^{n_1}$ and $M_2$ be a submanifold of an open set $V\subset\R^{n_2}$. If $\Phi\colon U\to V$ is a $\mathcal{C}^\infty$-immersion {\rm [\/}$\mathcal{C}^\infty$-submersion{\rm ]\/} with $\Phi(M_1)\subset M_2$, then the restriction $\Phi|_{M_1}\colon M_1\to M_2$ is a $\mathcal{C}^\infty$-immersion {\rm [\/}$\mathcal{C}^\infty$-submersion{\rm ]\/}.
\end{lemma}

Recall that given a non-empty set $X$ a family $\I\subset 2^X$ is said to be a {\it proper $\sigma$-ideal\/} iff the following conditions hold:
\begin{itemize}
\item[(i)] $X\not\in\I$;
\item[(ii)] if $A\in\I$ and $B\subset A$, then $B\in\I$;
\item[(iii)] if $A_k\in\I$ for $k\in\N$, then $\bigcup_{k=1}^\infty A_k\in\I$.
\end{itemize}
From now on we suppose that for each $m\in\N$ a family $\I_m$ forms a proper $\sigma$-ideal of subsets of $\R^m$ satisfying the following conditions:
\begin{itemize}
\item[(H$_0$)] $\{ 0\}\in\I_1$;
\item[(H$_1$)] if $\p$ is a $\mathcal{C}^\infty$-diffeomorphism defined on an open set $U\subset\R^m$ and $A\in\I_m$, then $\p(A\cap U)\in\I_m$;
\item[(H$_2$)] if $m,n\in\N$ and $A\in\I_{m+n}$, then $\{x\in\R^m:\, A[x]\not\in\I_n\}\in\I_m$;
\item[(H$_3$)] if $m,n\in\N$ and $A\in\I_n$, then $\R^m\times A\in\I_{m+n}$.
\end{itemize}
Note that by condition (H$_1$), non-empty open subsets of $\R^m$ do not belong to $\I_m$, whereas (H$_0$) and (H$_1$) imply that any countable subset of $\R^m$ is in $\I_m$.
\begin{remark}\label{Rem1}
The conditions (H$_0$)-(H$_3$) are satisfied in the following cases:
\begin{itemize*}
\item[(a)] when $\I_m$ consists of all first category subsets of $\R^m$, for $m\in\N$ (in this case (H$_2$) follows from the Kuratowski--Ulam theorem);
\item[(b)] when $\I_m$ consists of all Lebesgue measure zero subsets of $\R^m$, for $m\in\N$ (in this case (H$_2$) is just the classical Fubini theorem).
\end{itemize*}
More generally, let $\mu$ be any measure defined on all Borel subsets of $\R$ and satisfying both (H$_0$) and (H$_1$). Let also $$\mu_m=\underbrace{\mu\otimes\cdots\otimes\mu}_m$$ be the $m$th product measure and $\w\mu_m$ be the completion of $\mu_m$, for $m\in\N$. Then (H$_0$)-(H$_3$) are also satisfied in the two following cases:
\begin{itemize*}
\item[(c)] when $\I_m$ consists of all Borel subsets $A$ of $\R^m$ with $\mu_m(A)=0$ (condition (H$_1$) follows by induction from Fubini's theorem applied to the characteristic function of the Borel set $\p(A\cap U)$); 
\item[(d)] when $\I_m$ consists of all $\mu_m$-negligible subsets of $\R^m$, i.e., all $\w\mu_m$-measurable sets $A\subset\R^m$ with $\w\mu_m(A)=0$ (if $A\in\I_m$ then $A$ is contained in a~Borel set having measure $\mu_m$ zero, thus condition (H$_1$) follows as in the preceding case).
\end{itemize*}
\end{remark}

For an arbitrary $m$-manifold $M\subset\R^n$ ($m\leq n$) with an atlas $\mathcal{A}=\{(U_i,\p_i)\}_{i\in I}$ we define a proper $\sigma$-ideal $\I_M\subset 2^M$ by putting 
\begin{equation}\label{def_IM}
\I_M=\{A\subset M:\, \p_i(A\cap U_i)\in\I_m\mbox{ for each }i\in I\}.
\end{equation}
By condition (H$_1$), this definition does not depend on the particular choice of $\mathcal{A}$. Indeed, let $\{(V_j,\psi_j)\}_{j\in J}$ be another atlas of $M$, equivalent to $\mathcal{A}$. Fix any $A\in\I_M$ and $j\in J$. With the aid of Lindel\"of's theorem we choose a countable set $I_0\subset I$ such that $V_j\subset\bigcup_{i\in I_0}U_i$. For each $i\in I_0$ the mapping $\chi_i:=\psi_j\circ\p_i^{-1}$ is a $\mathcal{C}^\infty$-diffeomorphism on $\p_i(V_j\cap U_i)$ and since $B_i:=\p_i(A\cap V_j\cap U_i)\in\I_m$, we have $\psi_j(A\cap V_j\cap U_i)=\chi_i(B_i)\in\I_m$. Consequently, $\psi_j(A\cap V_j)=\bigcup_{i\in I_0}\psi_j(A\cap V_j\cap U_i)\in\I_m$. This shows that if $A\in\I_M$, then $\psi_j(A\cap V_j)\in\I_m$ for each $j\in J$. Analogously we obtain the reverse implication. Note that, by this definition, $\I_{\R^m}=\I_m$ for each $m\in\N$.

\begin{lemma}\label{dim}
Let $M_1$ be an $m_1$-dimensional submanifold of an $m_2$-manifold $M_2\subset\R^n$. Then
\begin{itemize}
\item[{\rm (a)\/}] $M_1\in\I_{M_2}$, provided that $m_1<m_2$;
\item[{\rm (b)\/}] $\I_{M_1}\subset\I_{M_2}$.
\end{itemize}
\end{lemma}
\begin{proof}
(a) By the submanifold property, we may choose an atlas $\mathcal{A}$ of $M_2$ such that $\p(M_1\cap U)=\p(U)\cap(\R^{m_1}\times\{0\})$ for each $(U,\p)\in\mathcal{A}$. Since (H$_0$) and (H$_3$) imply $\R^{m_1}\times\{0\}\in\I_{m_2}$, we get $\p(M_1\cap U)\in\I_{m_2}$, as desired.

(b) The case $m_1<m_2$ reduces to assertion (a). If $m_1=m_2$, then for every admissible chart of $M_2$ we have $\p(A\cap U)\in\I_{m_1}=\I_{m_2}$.
\end{proof}

We can prove the following strengthening of condition (H$_1$).
\begin{lemma}\label{111}
If $\Phi\colon M_1\to M_2$ is a $\mathcal{C}^\infty$-diffeomorphism between manifolds $M_1\subset\R^{n_1}$, $M_2\subset\R^{n_2}$, then for every $A\in\I_{M_1}$ we have $\Phi(A)\in\I_{M_2}$.
\end{lemma}
\begin{proof}
Let $\mathcal{A}_1=\{(U_i,\p_i)\}_{i\in I}$ and $\mathcal{A}_2=\{(V_j,\psi_j)\}_{j\in J}$ be atlases generating the differentiable structures of $M_1$ and $M_2$, respectively. Let also $m$ be the dimension of $M_1$ and $M_2$. Fix $j\in J$; we are to prove that $\psi_j(\Phi(A)\cap V_j)\in\I_m$. Choose a countable set $I_0\subset I$ with $A\subset\bigcup_{i\in I_0}U_i$ and for each $i\in I_0$ define a $\mathcal{C}^\infty$-diffeomorphism $\chi_i=\psi_j\circ\Phi\circ\p_i^{-1}$. Then
\begin{equation}\label{iii}
\psi_j(\Phi(A)\cap V_j)\subset\bigcup_{i\in I_0}\chi_i(\p_i(A\cap U_i)\cap {\rm Dom\/}(\chi_i)),
\end{equation}
where ${\rm Dom\/}(\chi_i)$ stands for the domain of $\chi_i$. Moreover, since $A\in\I_{M_1}$, we have $\p_i(A\cap U_i)\in\I_m$ thus (H$_1$) implies that the both sets in \eqref{iii} belong to $\I_m$.
\end{proof}

Conditions (H$_1$), (H$_2$) imply a general version of Fubini's theorem.
\begin{lemma}\label{222}
Let $M_1\subset\R^{n_1}$, $M_2\subset\R^{n_2}$ be manifolds. If $A\in\I_{M_1\times M_2}$, then $$\{x\in M_1:\, A[x]\not\in\I_{M_2}\}\in\I_{M_1}.$$
\end{lemma}
\begin{proof}
Let $\{(U_i,\p_i)\}_{i\in I}$ and $\{(V_j,\psi_j)\}_{j\in J}$ be arbitrary countable atlases generating the differentiable structures of $M_1$ and $M_2$, respectively. Since $A\in\I_{M_1\times M_2}$, for each $i\in I$, $j\in J$ we have $$B_{ij}:=(\p_i\times\psi_j)(A\cap (U_i\times V_j))\in\I_{m_1+m_2}.$$Moreover, $$B_{ij}=\{(\p_i(x),\psi_j(y))\in\R^{m_1+m_2}:\, x\in U_i\mbox{ and }y\in A[x]\cap V_j\}$$for $i\in I$, $j\in J$. Suppose, in search of a contradiction, that $$Z:=\{x\in M_1:\, A[x]\not\in\I_{M_2}\}\not\in\I_{M_1}.$$Then we may find $i_0\in I$ with $Z\cap U_{i_0}\not\in\I_{M_1}$. If for every $j\in J$ the set $$C_j:=\{x\in Z\cap U_{i_0}:\, A[x]\cap V_j\not\in\I_{M_2}\}$$belonged to $\I_{M_1}$, then we would have $$Z\cap U_{i_0}=\{x\in Z\cap U_{i_0}: A[x]\not\in\I_{M_2}\}=\bigcup_{j\in J}C_j\in\I_{M_1},$$which is not the case. Therefore, we may find $j_0\in J$ with $C_{j_0}\not\in\I_{M_1}$. Define $$B=\{(\p_{i_0}(x),\psi_{j_0}(y))\in\R^{m_1+m_2}:\, x\in Z\cap U_{i_0}\mbox{ and }y\in A[x]\cap V_{j_0}\}$$and note that $B\subset B_{i_0,j_0}$, whence $B\in\I_{m_1+m_2}$. However, $\p_{i_0}(C_{j_0})\not\in\I_{m_1}$ and for each $x\in C_{j_0}$ and $t=\p_{i_0}(x)$ we have $$B[t]=\psi_{j_0}(A[x]\cap V_{j_0})\not\in\I_{m_2}.$$This yields a contradiction with (H$_2$).
\end{proof}

\begin{lemma}\label{submersion2}
If $\Phi\colon M_1\to M_2$ is a $\mathcal{C}^\infty$-submersion between manifolds $M_1\subset\R^{n_1}$, $M_2\subset\R^{n_2}$, then for every $A\subset M_1$, $A\not\in\I_{M_1}$ we have $\Phi(A)\not\in\I_{M_2}$.
\end{lemma}
\begin{proof}
By Lindel\"of's theorem, there is a point $x_0\in M_1$ such that for every its neighborhood $U\subset M_1$ we have $A\cap U\not\in\I_{M_1}$. By the assumption, we may find admissible charts $(U,\p)$ and $(V,\psi)$ of $M_1$ and $M_2$, respectively, such that $x_0\in U$, $\Phi(x_0)\in V$, $\p(A\cap U)\not\in\I_{m_1}$ and the derivative of $\psi\circ\Phi\circ\p^{-1}$ at any point of $\p(U)$ is a surjection from $\R^{m_1}$ onto $\R^{m_2}$ ($m_1$, $m_2$ being the dimensions of $M_1$, $M_2$, respectively). Hence, obviously, $m_1\geq m_2$ and there is a sequence $1\leq i_1<\ldots<i_{m_2}\leq m_1$ such that $${\partial(\psi\circ\Phi\circ\p^{-1})\over\partial y_{i_1}\ldots\partial y_{i_{m_2}}}(\p(x_0))\not=0.$$By decreasing the neighborhood $U$, we may guarantee that the above condition holds true for every $x\in U$ in the place of $x_0$, and that the mapping $\psi\circ\Phi\circ\p^{-1}$ is defined on the whole $\p(U)$. Let $\psi\circ\Phi\circ\p^{-1}=(G_1,\ldots ,G_{m_2})$ and define a function $F=(F_1,\ldots ,F_{m_1})\colon\p(U)\to\R^{m_1}$ by the formula $$F_k(y)=\left\{\begin{array}{cl}G_j(y) & \mbox{if }k=i_j\mbox{ for some }j\in\{ 1,\ldots ,m_2\},\\
y_k & \mbox{otherwise.}\end{array}\right.$$Then for each $y\in\p(U)$ we have $$\left|{\partial F\over\partial y_1\ldots\partial y_{m_1}}(y)\right|=\left|{\partial(\psi\circ\Phi\circ\p^{-1})\over\partial y_{i_1}\ldots\partial y_{i_{m_2}}}(y)\right|\not=0,$$thus, decreasing $U$ as required, we may assume that $F$ is a $\mathcal{C}^\infty$-diffeomorphism. Enumerating the coordinates we may also modify $F$ in such a way that it is still a $\mathcal{C}^\infty$-diffeomorphism and
\begin{equation}\label{FpA}
F(\p(A\cap U))\subset (\psi\circ\Phi\circ\p^{-1})(\p(A\cap U))\times\R^{m_1-m_2} .
\end{equation}
In view of $\p(A\cap U)\not\in\I_{m_1}$, condition (H$_1$) yields $F(\p(A\cap U))\not\in\I_{m_1}$, whence \eqref{FpA} and (H$_3$) imply $\psi(\Phi(A\cap U))\not\in\I_{m_2}$. Therefore, $\Phi(A\cap U)\not\in\I_{M_2}$, since $\psi$ is an admissible chart of $M_2$ defined on $\Phi(U)$.
\end{proof}

In a similar manner we obtain the next lemma.
\begin{lemma}\label{immersion}
If $\Phi\colon M_1\to M_2$ is a $\mathcal{C}^\infty$-immersion between manifolds $M_1\subset\R^{n_1}$, $M_2\subset\R^{n_2}$, then for every $A\in\I_{M_1}$ we have $\Phi(A)\in\I_{M_2}$.
\end{lemma}

From now on, let $n\geq 2$ be a fixed natural number and $\langle\cdot\vert\cdot\rangle$ be an arbitrary inner product in $\R^n$ inducing a norm which we denote by $\|\cdot\|$. For any set $A$ we define $A^\ast=A\setminus\{ 0\}$, where the meaning of $0$ is clear from the context. Let $\bot$ be the set of all pairs of orthogonal vectors from $\R^n$. Then $\bot^\ast=F^{-1}(0)$, where $F\colon(\R^n\times\R^n)^\ast\to\R$ is given by $F(x,y)=\langle x\vert y\rangle$. Since $0$ is a regular value of $F$, it follows from \cite[Theorem 9.11]{tu} that $\bot^\ast$ forms a $(2n-1)$-manifold (being also a regular submanifold of $(\R^n\times\R^n)^\ast$).

We may therefore make it precise what being \lq\lq negligible\rq\rq\ in $\bot$ means. Namely, we say that a set $Z\subset\bot$ has this property iff $Z\in\I_{\bot^\ast}$ and we will then write simply $Z\in\I_\bot$. We are now ready to formulate our main result which we shall prove in the last section. For notational convenience, if $M$ is a manifold and some property, depending on a variable $x$, holds true for all $x\in M\setminus A$ with $A\in\I_M$, then we write that it holds $\I_M$-(a.e.).
\begin{theorem}
Let $(G,+)$ be an Abelian group. If a function $f\colon\R^n\to G$ satisfies $f(x+y)=f(x)+f(y)$ $\I_\bot${\rm -(a.e.)\/},
then there is a unique orthogonally additive function $g\colon\R^n\to G$ such that $f(x)=g(x)$ $\I_n${\rm -(a.e.)\/}.
\end{theorem}

\begin{remark}\label{Rem2}
According to Remark \ref{Rem1}, the above theorem works whenever the ideal $\I_\bot$ is defined via formula \eqref{def_IM} for $(\I_m)_{m=1}^\infty$ being one of the sequences of ideals described in (a)-(d). 

In case (a) the ideal $\I_\bot$ consists of all first category subsets of $\bot^\ast$, regarded as a~topological subspace of the Euclidean space $\R^{2n}$. 

In case (b) the ideal $\I_\bot$ consists of all Lebesgue measure zero subsets of $\bot^\ast$. Recall that the Lebesgue measure on any regular submanifold $M$ of $\R^n$ is defined with the aid of the formula $$\mu_M(A)=\int_{\p(A)}\bigl|(\p^{-1})^\prime(\boldsymbol{x})\bigr|\,\mathrm{d}\boldsymbol{x},$$postulated for any admissible chart $(U_\p,\p)$ of $M$ and any set $A\subset M$ such that $A\subset U_\p$ and $\p(A)\subset\R^m$ is Lebesgue measurable. 

Further examples are produced by the ideals $\I_m$ described in (c)-(d), in Remark \ref{Rem1}. For instance, one may start with the $\alpha$-dimensional Hausdorff measure $\mathscr{H}^\alpha$ (for some $0<\alpha<1$) defined on all Borel subsets (or on all Hausdorff measurable subsets) of $\R$ and, by using formula \eqref{def_IM}, induce a~corresponding ideal $\I_\bot$. However, this ideal will not be the same as the ideal of all Borel (Hausdorff measurable) sets $A\subset\bot^\ast$ with $\mathscr{H}^{\alpha(2n-1)}(A)=0$ (the $\alpha(2n-1)$-dimensional Hausdorff measure on the metric space $\bot^\ast$), since the product measure $\mathscr{H}^\alpha\otimes\mathscr{H}^\alpha$ need not be the Hausdorff measure $\mathscr{H}^{2\alpha}$ (consult \cite[\S 3.1]{rogers} and the references therein). This leads to the following question: Let $0<\alpha<1$. Is our Theorem true in the case where $\I_\bot$ is the set of all Borel (Hausdorff measurable) sets $A\subset\bot^\ast$ with $\mathscr{H}^{\alpha(2n-1)}(A)=0$ and $\I_n$ is replaced by the ideal of all Borel (Hausdorff measurable) sets $B\subset\R^n$ with $\mathscr{H}^{\alpha n}(B)=0$?
\end{remark}

Before proceeding to further lemmas, let us note some preparatory observations. For any $x\in\R^n$ define $$P_x=\{y\in\R^n:\, (x,y)\in\bot\},$$ which obviously forms an $(n-1)$-manifold diffeomorphic to $\R^{n-1}$, provided $x\not=0$. We will need to \lq\lq smoothly\rq\rq\ identify the hyperplanes $P_x$, for different $x$'s, with one \lq\lq universal\rq\rq\ space $\R^{n-1}$. By virtue of the Hairy Sphere Theorem, it is impossible to do for all $x\in (\R^n)^\ast$ in the case where $n$ is odd. Nevertheless, it is an easy task when considering only the set of vectors for which one fixed coordinate is non-zero, e.g. the set $$X:=\R^{n-1}\times\R^\ast.$$

Namely, for an arbitrary $x\in X$ the vectors $x,e_1,\ldots ,e_{n-1}$ are linearly independent, where $e_i$ stands for the $i$th vector from the canonical basis of $\R^n$. Let $\mathcal{B}(x)=(y_i(x))_{i=0}^{n-1}$ be an orthonormal basis of $\R^n$ with $y_0(x)=x/\|x\|$, produced by the Gram-Schmidt process applied to the sequence $(x,e_1,\ldots ,e_{n-1})$. Define $\psi_x\colon\R^n\to\R^n$ to be the mapping which to every $z\in\R^n$ assigns its coordinates with respect to $\mathcal{B}(x)$, i.e. $\psi_x(z)=\boldsymbol{Y}(x)^{-1}z$, where $$\boldsymbol{Y}(x)=\left({x\over\|x\|},y_1(x),\ldots ,y_{n-1}(x)\right)$$is the matrix formed from the column vectors. Define also $\Phi\colon X\times\R^n\to X\times\R^n$ by $\Phi(x,z)=(x,\psi_x(z))$. Plainly, $\Phi$ is a $\mathcal{C}^\infty$-mapping and its inverse $\Phi^{-1}(x,y)=(x,Y(x)y)$ is $\mathcal{C}^\infty$ as well. Therefore, $\Phi$ is a $\mathcal{C}^\infty$-diffeomorphism. Moreover, by the definition of $\psi_x$, the restriction $\psi_x|_{P_x}$ maps $P_x$ onto $\{ 0\}\times\R^{n-1}$, hence we have
\begin{equation}\label{Pxx}
\Phi^{-1}\left(X\times(\{0\}\times\R^{n-1})\right)=\{ (x,z)\in\bot^\ast :\, x\in X\}=:\bot^\prime .
\end{equation}
Making use of \cite[Theorem 11.20]{tu} and an easy fact that the restriction of a $\mathcal{C}^\infty$ mapping to a submanifold of its domain is $\mathcal{C}^\infty$ again\footnote{In the sequel, we will be using these two assertions without explicit mentioning.}, we infer by \eqref{Pxx} that $\Phi|_{\bot^\prime}$ yields a $\mathcal{C}^\infty$-diffeomorphism between $\bot^\prime$ and $X\times (\{0\}\times\R^{n-1})$.

Consequently, if a function $h\colon\R^n\to G$ satisfies $h(x+y)=h(x)+h(y)$ $\I_\bot$-(a.e.), then with the notation $$Z(h):=\{ (x,y)\in\bot^\ast:\, h(x+y)\not=h(x)+h(y)\}$$it follows from Lemmas \ref{111} and \ref{222} that $$\left\{ x\in X:\,\{\psi_x(z):\, (x,z)\in Z(h)\}\not\in\I_{\{0\}\times\R^{n-1}}\right\}\in\I_X.$$Since $P_x\sim\{0\}\times\R^{n-1}$, by the mapping $\psi_x|_{P_x}$ for $x\in X$, we infer that the set $$D(h):=\{x\in X:\, h(x+y)=h(x)+h(y)\,\,\I_{P_x}\mbox{-(a.e.)}\}$$satisfies $X\setminus D(h)\in\I_X$. For any $x\in\R^n$ put $$E_x(h)=\{ y\in P_x:\, h(x+y)=h(x)+h(y)\} ;$$then $P_x\setminus E_x(h)\in\I_{P_x}$, provided $x\in D(h)$.

We end this section with a lemma, which will be useful in the \lq\lq odd\rq\rq\ part of the proof of our Theorem. Despite it will be applied only in the case $n=2$, we present it in full generality, since the lemma seems to be interesting independently on the problem considered. Let $S^{n-1}$ be the unit sphere of the normed space $(\R^n,\|\cdot\|)$. Since the function $F\colon\R^n\to\R$ given by $F(x)=\|x\|^2$ is $\mathcal{C}^\infty$ with the regular value $1$ and $S^{n-1}=F^{-1}(1)$, we infer that $S^{n-1}$ is an $(n-1)$-manifold.
\begin{lemma}\label{sfera}
If $A\in\I_{S^{n-1}}$, then there exists an orthogonal basis $(x_1,\ldots ,x_n)$ of $\R^n$ such that $x_i\in S^{n-1}\setminus A$ for each $i\in\{ 1,\ldots ,n\}$.
\end{lemma}
\begin{proof}
It is enough to prove the assertion in the case where $\langle\cdot\vert\cdot\rangle$ is the standard inner product in $\R^n$, since between any two inner product structures in $\R^n$ there is a linear isometry, which yields a $\mathcal{C}^\infty$-diffeomorphism between their unit spheres.

Consider the group $\mathrm{GL}(n)$ of $n\times n$ real matrices with non-zero determinant. It may be identified with an open subset of $\R^{n^2}$ and hence - it is an $n^2$-manifold. It is well-known that the orthogonal group $$\mathrm{O}(n)=\{\boldsymbol{A}\in\mathrm{GL}(n):\, \boldsymbol{A}\boldsymbol{A}^T=\boldsymbol{I}_n\}$$forms a submanifold of $\mathrm{GL}(n)$ and its dimension equals $n(n-1)/2$ (see \cite[\S 3.5.5C]{abraham}). For any $i\in\{ 1,\ldots ,n\}$ let $\pi_i\colon\mathrm{O}(n)\to S^{n-1}$ be given by $\pi_i(\boldsymbol{A})=\boldsymbol{A}e_i$ (which is nothing else but the $i$th column vector of $\boldsymbol{A}$). Then $\pi_i$ is the restriction of the mapping $\oo{\pi}_i\colon\mathrm{GL}(n)\to\R^n$ defined by the formula analogous to the previous one. Since $$\mathrm{D}\oo{\pi}_i(\boldsymbol{A})\boldsymbol{B}=\boldsymbol{B}e_i\quad\mbox{for }\boldsymbol{A}\in\mathrm{GL}(n),\,\, \boldsymbol{B}\in\R^{n^2},$$the derivative $\mathrm{D}\oo{\pi}_i(\boldsymbol{A})$ is onto for any $\boldsymbol{A}\in\mathrm{GL}(n)$, thus $\oo{\pi}_i$ is a $\mathcal{C}^\infty$-submersion. By Lemma \ref{submersion}, $\pi_i$ is a $\mathcal{C}^\infty$-submersion as well.

Now, suppose on the contrary that each orthonormal basis of $\R^n$ has at least one entry belonging to $A$. In other words, for each $\boldsymbol{A}\in\mathrm{O}(n)$ there is $i\in\{ 1,\ldots ,n\}$ with $\pi_i(\boldsymbol{A})\in A$, i.e. $$\mathrm{O}(n)=\bigcup_{i=1}^n\pi_i^{-1}(A).$$Therefore, for a certain $i\in\{ 1,\ldots ,n\}$ we would have $\pi_i^{-1}(A)\not\in\I_{\mathrm{O}(n)}$. However, $A=\pi_i(\pi_i^{-1}(A))\in\I_{S^{n-1}}$, which contradicts the assertion of Lemma \ref{submersion2}, as $\pi_i$ is a $\mathcal{C}^\infty$-submersion.
\end{proof}

\section{Proof of the Theorem}For the uniqueness part of our Theorem suppose that there are two orthogonally additive functions $g_1$ and $g_2$ equal to $f$ $\I_n$-(a.e.). By the general form \eqref{1} of orthogonally additive mappings, we see that both $g_1$ and $g_2$ satisfy the Fr\'echet functional equation $\Delta_y^3g(x)=0$, thus arguing as in the proof of the uniqueness part of \cite[Theorem 1]{ger0}, or making use of \cite[Lemma 17.7.1]{kuczma}, we get $g_1=g_2$.

The proof of existence relies on some ideas from \cite{baron_ratz} and \cite{ratz}. Assume $G$ and $f$ are as in the Theorem. We start with the following trivial observation.
\begin{lemma}\label{L0}
The functions $f_1,f_2\colon\R^n\to G$ given by $$f_1(x)=f(x)-f(-x)\quad\mbox{and}\quad f_2(x)=f(x)+f(-x)$$satisfy $$f_1(x+y)=f_1(x)+f_1(y)\quad\mbox{and}\quad f_2(x+y)=f_2(x)+f_2(y)\quad\I_\bot\mbox{\rm -(a.e.).}$$
\end{lemma}
In the sequel we will be using hypothesis (H$_0$)-(H$_3$) and Lemmas \ref{dim}-\ref{222} without explicit mentioning.

For $k,m\in\N$ with $2\leq k\leq m$ we define $\Or(k,m)$ as the set of all $k$-tuples of mutually orthogonal (with respect to the usual scalar product) vectors from $\R^m$ with at most one of them being zero. Put $$\mathcal{R}_{k,m}=\{(x^{(1)},\ldots ,x^{(k)})\in (\R^m)^k:\, x^{(i)}=0\mbox{ for at most one }i=1,\ldots ,k\} .$$Then $\Or(k.m)=F^{-1}(0)$, where $F\colon\mathcal{R}_{k,m}\to\R^{{k(k-1)\over 2}}$ is given by $$\begin{array}{r}F(x^{(1)},\ldots ,x^{(k)})=(\langle x^{(1)}\vert x^{(2)}\rangle , \langle x^{(1)}\vert x^{(3)}\rangle ,\ldots , \langle x^{(1)}\vert x^{(k)}\rangle ,\,\,\\
\langle x^{(2)}\vert x^{(3)}\rangle ,\ldots , \langle x^{(2)}\vert x^{(k)}\rangle ,\,\,\\
\vdots\qquad\,\,\,\,\\
\langle x^{(k-1)}\vert x^{(k)}\rangle ).\end{array}$$
Since $0$ is a regular value of $F$, \cite[Theorem 9.11]{tu} implies that $\Or(k,m)$ is a submanifold of $\R^{km}$ with dimension $km-{1\over 2}k(k-1)$. In particular, $\Or(2,n)=\bot^\ast$.
\begin{lemma}\label{O(k,m)}
Let $k\in\N$, $k\geq 2$ and let $A\subset\Or(2,k)$ be a set such that $$\{(x^{(1)},\ldots ,x^{(k)})\in\Or(k,k):\, (x^{(1)},x^{(2)})\in A\}\in\I_{\Or(k,k)}.$$Then $A\in\I_{\Or(2,k)}$.
\end{lemma}
\begin{proof}
Denote the above subset of $\Or(k,k)$ by $B$. We may clearly assume that for each $(x^{(1)},x^{(2)})\in A$ we have $x^{(1)}\not=0\not=x^{(2)}$. For $i,j\in\{ 1,\ldots ,k\}$ define $$D_{ij}=\Biggl\{(x^{(1)},x^{(2)})\in\Or(2,k):\,\det\Biggl(\begin{array}{cc}x_i^{(1)} & x_j^{(1)}\\ x_i^{(2)} & x_j^{(2)}\end{array}\Biggr)\not=0\Biggr\},$$ $$B_{ij}=\{(x^{(1)},\ldots ,x^{(k)})\in B:\, (x^{(1)},x^{(2)})\in D_{ij}\}.$$We will show that
\begin{equation}\label{A_and_B}
A=\bigcup_{\substack{i,j=1\\ i\not=j}}^k(A\cap D_{ij})\quad\mbox{and}\quad B=\bigcup_{\substack{i,j=1\\ i\not=j}}^kB_{ij}.
\end{equation}
For the former equality suppose that for some $(x^{(1)},x^{(2)})\in A$ and each pair of indices  $1\leq i,j\leq k$, $i\not=j$, we have
\begin{equation}\label{det=0}
\det\Biggl(\begin{array}{cc}x_i^{(1)} & x_j^{(1)}\\ x_i^{(2)} & x_j^{(2)}\end{array}\Biggr)=0.
\end{equation}
Then for each $1\leq i\leq k$ we have $x_i^{(1)}=0$ if and only if $x_i^{(2)}=0$. Indeed, choosing any $1\leq j\leq k$ such that $x_j^{(1)}\not=0$ we see from \eqref{det=0} that $x_i^{(1)}=0$ implies $x_i^{(2)}=0$; the reverse implication holds by symmetry. Now, let $1\leq i_1<\ldots <i_\ell\leq k$ be the indices of all non-zero coordinates of $x^{(1)}$ (and $x^{(2)}$). For each pair of $1\leq i,j\leq k$ one of the rows of the determinant in \eqref{det=0} is a multiple of the other. Applying this observation consecutively for the pairs $(i_1,i_2),(i_2,i_3),\ldots ,(i_{\ell-1},i_\ell)$ we infer that $x^{(1)}$ and $x^{(2)}$ are parallel. Since they are also orthogonal, one of them should be zero which is the case we have excluded. The former equality in \eqref{A_and_B} is thus proved, and its easy consequence is the latter one.

We are now to show that $A\cap D_{ij}\in\I_{\Or(2,k)}$ for each pair of indices $i,j\in\{ 1,\ldots ,k\}$ with $i\not=j$. So, fix any such pair and assume that $i<j$. Then for every $(x^{(1)},x^{(2)})\in D_{ij}$ the vectors: $$x^{(1)},x^{(2)},e_1,\ldots ,e_{i-1},e_{i+1},\ldots ,e_{j-1},e_{j+1},\ldots ,e_k$$form a basis of $\R^k$. Let $$\mathcal{B}(x^{(1)},x^{(2)})=\bigl(y_i(x^{(1)},x^{(2)})\bigr)_{i=1}^k$$be an orthonormal basis produced by the Gram-Schmidt process applied to that sequence of vectors. Since $x^{(1)}$ and $x^{(2)}$ are orthogonal, we have $$y_1(x^{(1)},x^{(2)})={x^{(1)}\over\|x^{(1)}\|}\quad\mbox{and}\quad y_2(x^{(1)},x^{(2)})={x^{(2)}\over\|x^{(2)}\|}\, .$$For $(x^{(1)},x^{(2)})\in D_{ij}$ define $\vartheta_{x^{(1)},x^{(2)}}\colon\R^k\to\R^k$ as the mapping which to every $z\in\R^k$ assigns its coordinates with respect to $\mathcal{B}(x^{(1)},x^{(2)})$, i.e. $$\vartheta_{x^{(1)},x^{(2)}}(z)=\boldsymbol{Y}(x^{(1)},x^{(2)})^{-1}z,$$where $$\boldsymbol{Y}(x^{(1)},x^{(2)})=\left({x^{(1)}\over\|x^{(1)}\|},{x^{(2)}\over\|x^{(2)}\|},y_3(x^{(1)},x^{(2)}),\ldots ,y_k(x^{(1)},x^{(2)})\right)$$is formed from the column vectors. Obviously, every $z$ belonging to the orthogonal complement $V(x^{(1)},x^{(2)})^\bot$ of the subspace spanned by $x^{(1)}$ and $x^{(2)}$ is mapped onto a certain vector of the form $(0,0,t_3,\ldots ,t_k)$ which may be naturally identified with an element of $\R^{k-2}$. Hence, we get a linear isomorphism $\gamma_{x^{(1)},x^{(2)}}\colon V(x^{(1)},x^{(2)})^\bot\to\R^{k-2}$ and we may define a mapping $$\Gamma\colon\{(x^{(1)},\ldots ,x^{(k)})\in\Or(k,k):\, (x^{(1)},x^{(2)})\in D_{ij}\}\to\bigl(\Or(2,k)\cap D_{ij}\bigr)\times\Or(k-2,k-2)$$by the formula $$\Gamma(x^{(1)},\ldots ,x^{(k)})=\bigl((x^{(1)},x^{(2)}),(\gamma_{x^{(1)},x^{(2)}}(x^{(3)}),\ldots ,\gamma_{x^{(1)},x^{(2)}}(x^{(k)}))\bigr).$$The definition is well-posed, since $\vartheta_{x^{(1)},x^{(2)}}$, and hence also $\gamma_{x^{(1)},x^{(2)}}$, is an isometry for each $(x^{(1)},x^{(2)})\in D_{ij}$. Moreover, it is easily seen that $\Gamma$ is a $\mathcal{C}^\infty$-diffeomorphism (the formulas of the Gram-Schmidt procedure are $\mathcal{C}^\infty$).

It easily follows from $B\in\I_{\Or(k,k)}$ that $B_{ij}$ belongs to the corresponding ideal of subsets of $$\{(x^{(1)},\ldots ,x^{(k)})\in\Or(k,k):\, (x^{(1)},x^{(2)})\in D_{ij}\},$$thus $\Gamma(B_{ij})$ belongs to the ideal corresponding to $(\Or(2,k)\cap D_{ij})\times\Or(k-2,k-2)$. Finally, observe that $$\Gamma(B_{ij})=(A\cap D_{ij})\times\Or(k-2,k-2),$$which yields $A\cap D_{ij}\in\I_{\Or(2,k)\cap D_{ij}}$ and hence also $A\cap D_{ij}\in\I_{\Or(2,k)}$.
\end{proof}

\begin{lemma}\label{L1}
If an odd function $h\colon\R^n\to G$ satisfies $h(x+y)=h(x)+h(y)$ $\I_\bot${\rm -(a.e.)\/}, then there is an additive function $b\colon\R^n\to G$ such that $h(x)=b(x)$ $\I_n${\rm -(a.e.)\/}.
\end{lemma}
\begin{proof}
Due to some isometry formalities, we may suppose $\langle\cdot\vert\cdot\rangle$ to be the standard inner product in $\R^n$.

Define $$W=\{ x=(x_1,\ldots ,x_n)\in\R^n:\,x_i=0\mbox{ for some }i\}$$and $$S_+^{n-1}=\{x=(x_1,\ldots ,x_n)\in S^{n-1}:\, x_n>0\} .$$Since $S_+^{n-1}$ is an open subset of $S^{n-1}$, it is an $(n-1)$-manifold. For any $x\in S_+^{n-1}$ let $$T_x=\{(\lambda,y)\in\R^\ast\times P_x^\ast :\,\lambda^2=\|y\|^2\}.$$Define a~map $\oo{\Phi}_x\colon \R^\ast\times P_x^\ast\to\R^n\times\R^n$ by
\begin{equation}\label{PhiX}
\oo{\Phi}_x(\lambda ,y)=\Bigl(\lambda x+y,{\|y\|^2\over\lambda}x-y\Bigr),
\end{equation}
and set $\Phi_x=\oo{\Phi}_x|_{(\R^\ast\times P_x^\ast)\setminus T_x}$. Let also $Q(x)=\Phi_x\bigl((\R^\ast\times P_x^\ast)\setminus T_x\bigr)\subset\bot^\ast$. We are going to show that for every $x\in P:=S_+^{n-1}\setminus W$ the set $Q(x)$ forms a submanifold of $\bot^\ast$.

At the moment, let $x\in S_+^{n-1}$. For brevity, denote $\mu=\mu(\lambda,y)=\|y\|^2/\lambda$. It is easily seen that for each $(t,u)=(\lambda x+y,\mu x-y)\in Q(x)$ all four vectors: $t$, $u$, $x$, $y$ belong to the subspace $V(t,x)$ of $\R^n$ spanned by $t$ and $x$. Choose an arbitrary non-zero vector $z(t,x)\in V(t,x)$, orthogonal to $x$. Then $z(t,x)$ is collinear with $y$, hence the equality $t=\lambda x+y$ represents $t$ in terms of the basis $(x,z(t,x))$ of $V(t,x)$. Therefore, $\lambda$ and $y$ are uniquely determined by $t$, which proves that $\Phi_x$ is injective.

In order to show that $\Phi_x^{-1}$ is continuous fix an arbitrary $(t,u)\in Q(x)$. Now, put $z(t,x)=\langle t\vert x\rangle x-t$; then $(x,z(t,x))$ is an orthogonal basis of $V(t,x)$. Since $t=\lambda x+y$ for certain $\lambda\in\R^\ast$ and $y\in P_x^\ast$, we have $t=\lambda x+\alpha z(t,x)$ for some $\alpha\in\R$, whence we find that $\lambda=\langle t\vert x\rangle$ and $y=t-\langle t\vert x\rangle x$. We have thus shown that $\Phi_x$ is a homeomorphism.

Now, fix $x\in P$. We shall prove that $\Phi_x$ is a $\mathcal{C}^\infty$-immersion. To this end put $$V_x=\Bigl\{(\lambda,y)\in\R^\ast\times (\R^n)^\ast :\,\lambda=\langle x\vert y\rangle\pm\sqrt{\langle x\vert y\rangle^2+\|y\|^2}\Bigr\}$$and define a mapping $\hat{\Phi}_x\colon (\R^\ast\times (\R^n)^\ast)\setminus V_x\to\R^n\times\R^n$ by the formula analogous to \eqref{PhiX}. Then $(\R^\ast\times P_x^\ast)\setminus T_x$ is a submanifold of $(\R^\ast\times (\R^n)^\ast)\setminus V_x$. Let $(\lambda ,y)\in(\R^\ast\times (\R^n)^\ast)\setminus V_x$. If we show that the derivative $\mathrm{D}\hat{\Phi}_x(\lambda ,y)$ is injective, then, in view of Lemma \ref{submersion}, we will be done. Since $$\mathrm{D}\hat{\Phi}_x(\lambda ,y)=\left(\begin{array}{c|c}\begin{array}{c}x_1\\ x_2\\ \vdots\\ x_n\end{array} & \begin{array}{cccc}1 & 0 & \ldots & 0\\ 0 & 1 & \ldots & 0\\ \vdots & \vdots & \ddots & \vdots\\ 0 & 0 & \ldots & 1\end{array}\\ \hline\vspace{4pt}
_{\displaystyle{{\partial (\mu x-y)\over\partial\lambda}}} & _{\displaystyle{{\partial (\mu x-y)\over\partial y}}}\end{array}\right),$$we immediately get that $\mathrm{rank}\,\mathrm{D}\hat{\Phi}_x(\lambda ,y)\geq n$, where the equality occurs only if the first column vector is a linear combination of the remaining $n$ column vectors with coefficients $x_1,\ldots ,x_n$. However, this would imply that for each $i\in\{ 1,\ldots ,n\}$ we have $${\partial (\mu x_i-y_i)\over\partial\lambda}=\sum_{j=1}^nx_j{\partial (\mu x_i-y_i)\over\partial y_j},$$i.e. $$\lambda^2-2\langle x\vert y\rangle\lambda-\|y\|^2=0,$$which is not the case, since $(\lambda,y)\not\in V_x$. As a result, we obtain $\mathrm{rank}\,\mathrm{D}\hat{\Phi}_x(\lambda ,y)=n+1$, thus $\mathrm{D}\hat{\Phi}_x(\lambda ,y)$ is injective.

We have shown that $\Phi_x$ is an embedding (i.e. homeomorphic $\mathcal{C}^\infty$-immersion) of $(\R^\ast\times P_x^\ast)\setminus T_x$ into $\bot^\ast$. By virtue of \cite[Theorem 11.17]{tu}, its image $Q(x)$ is a submanifold of $\bot^\ast$.

Observe that the manifolds $Q(x)$, for $x\in P$, are $\mathcal{C}^\infty$-diffeomorphic each to others. Indeed, by the remarks following the statement of our Theorem, for each $x\in X$ the function $\Psi_x\colon\R^\ast\times P_x^\ast\to\R^\ast\times (\R^{n-1})^\ast$ defined by the formula
\begin{equation}\label{psipsi}
\Psi_x(\lambda ,y)=(\lambda ,\w{\psi}_x(y)),
\end{equation}
where $\psi_x(y)=\boldsymbol{Y}(x)^{-1}y$ is defined as earlier and the tilde operator deletes the first coordinate (which equals $0$ for $y\in P_x$), is a~$\mathcal{C}^\infty$-diffeomorphism. Moreover, $\Psi_x$ maps $(\R^\ast\times P_x^\ast)\setminus T_x$ onto the set $$U:=\{(\lambda,y)\in\R^\ast\times (\R^{n-1})^\ast :\,\lambda^2\not=\|y\|^2\},$$which follows from the fact that $\w{\psi}_x$ is an isometry. Therefore, for each $x,y\in P$, the mapping $\Phi_y\circ\Psi_y^{-1}\circ\Psi_x\circ\Phi_x^{-1}$ yields a $\mathcal{C}^\infty$-diffeomorphism between $Q(x)$ and $Q(y)$. So, we pick any $x_0\in P$ and we regard the set $Q:=Q(x_0)$ as a \lq\lq model\rq\rq\ manifold for all $Q(x)$'s.

Define $$\bot^{(1)}=\{ (t,u)\in\bot^\ast :\, t_n+u_n\not=0,\, t\not=0,\, u\not=0\,\,\mbox{ and }\,\,\|t\|\not=\|u\|\}$$ (which is an open subset, and hence it is a~submanifold, of $\bot^\ast$) and observe that
\begin{equation}\label{QS1}
\bot^{(1)}=\bigcup_{x\in S_+^{n-1}}Q(x).
\end{equation}
In fact, for any $(t,u)\in\bot^{(1)}$ put
\begin{equation}\label{xtu}
x=\mathrm{sgn}(t_n+u_n){t+u\over\|t+u\|}.
\end{equation}
Then $x\in S_+^{n-1}$ and $(t,u)\in Q(x)$. Indeed, if we choose any $y_0\in P_x^\ast\cap V(t,u)$ with $\|y_0\|=1$ (which is unique up to a sign), then $t$ and $u$ are represented in terms of the basis $(x,y_0)$ of $V(t,u)$ as follows: $$t=\langle t\vert x\rangle x+\langle t\vert y_0\rangle y_0\,\,\mbox{ and }\,\,u=\langle u\vert x\rangle x+\langle u\vert y_0\rangle y_0,$$and we have
\begin{equation*}
\langle t\vert y_0\rangle=\langle t+u\vert y_0\rangle-\langle u\vert y_0\rangle=\pm\|t+u\|\langle x\vert y_0\rangle-\langle u\vert y_0\rangle=-\langle u\vert y_0\rangle .
\end{equation*}
Hence, after substitution $\lambda=\langle t\vert x\rangle$ and $y=\langle t\vert y_0\rangle y_0$, we obtain $t=\lambda x+y$ and $u=\langle u\vert x\rangle x-y$. The coefficient $\langle u\vert x\rangle$ equals $\|y\|^2/\lambda$, since $\langle t\vert u\rangle=\langle x\vert y\rangle=0$. Moreover, $\lambda\not=0$, $y_0\not=0$, and it follows from $\|t\|\not=\|u\|$ that $\lambda^2\not=\langle u\vert x\rangle^2=\|y\|^4/\lambda^2$, which gives $\lambda^2\not=\|y\|^2$. Consequently, $(t,u)\in Q(x)$ and thus we have proved the inclusion \lq\lq $\subseteq$\rq\rq . The reverse inclusion is a straightforward calculation.

We shall now prove that the mapping $\Lambda\colon S_+^{n-1}\times U\to\bot^{(1)}$ defined by $$\Lambda(x,\lambda ,y)=\Phi_x\circ\Psi_x^{-1}(\lambda,y)$$is a $\mathcal{C}^\infty$-diffeomorphism.

First, in view of \eqref{QS1}, it is easily seen that the image of $\Lambda$ is $\bot^{(1)}$. According to the definition, $\Lambda$ is $\mathcal{C}^\infty$. Moreover, for each $(t,u)=\Phi_x\bigl(\lambda ,\w{\psi}_x^{-1}(y)\bigr)\in Q(x)$ we have
\begin{equation}\label{QS2}
\Bigl(\lambda+{\|\w{\psi}_x^{-1}(y)\|^2\over\lambda}\Bigr)x=t+u,
\end{equation}
which, jointly with the fact that $x\in S_+^{n-1}$, uniquely determines $x$. By the injectivity of $\Phi_x$, we infer that $\lambda$ and $y$ are then uniquely determined by $t$ and $u$ as well. Therefore, $\Lambda$ is injective.

In order to get a formula for $\Lambda^{-1}$, observe that for each $(t,u)=\Phi_x\bigl(\lambda ,\w{\psi}_x^{-1}(y)\bigr)\in\bot^{(1)}$ equality \eqref{QS2} yields \eqref{xtu}. This means that $x$ is expressed as a function of $t$ and $u$, which is $\mathcal{C}^\infty$ on both components of the set $\bot^{(1)}$. By the formula for $\Phi_x^{-1}$, we get $$\lambda=\mathrm{sgn}(t_n+u_n){\langle t\vert t+u\rangle\over\|t+u\|}\,\,\mbox{ and }\,\, y=\w{\psi}_x\Bigl(t-{\langle t\vert t+u\rangle\over\|t+u\|^2}(t+u)\Bigr),$$and since the value of $\w{\psi}_x$ at a given point is a $\mathcal{C}^\infty$ function of $x$, we infer that $\Lambda^{-1}$ is $\mathcal{C}^\infty$. Consequently, $\Lambda$ is a $\mathcal{C}^\infty$-diffeomorphism.

Let $\chi\colon\bot^{(1)}\to S_+^{n-1}\times Q$ be given by $$\chi=(\mathrm{id}_{S_+^{n-1}}\times\oo{\Phi}_{x_0})\circ (\mathrm{id}_{S_+^{n-1}}\times\Psi_{x_0}^{-1})\circ\Lambda^{-1};$$then $\chi$ is a $\mathcal{C}^\infty$-diffeomorphism. Since $Z(h)\in\I_\bot$ and $\bot^{(1)}$ is an open subset of $\bot^\ast$, we have $Z(h)\cap\bot^{(1)}\in\I_{\bot^{(1)}}$. Therefore,
\begin{equation}\label{QS3}
\{ x\in S_+^{n-1}:\,\chi(Z(h)\cap\bot^{(1)})[x]\not\in\I_Q\}\in\I_{S_+^{n-1}}.
\end{equation}
Let $x\in S_+^{n-1}$. For any $q\in Q$ we have $$q\in \chi(Z(h)\cap\bot^{(1)})[x] \Longleftrightarrow (x,q)\in\chi(Z(h)\cap\bot^{(1)}) \Longleftrightarrow \chi^{-1}(x,q)\in Z(h)\cap\bot^{(1)}.$$ Plainly, $$\chi^{-1}=\Lambda\circ (\mathrm{id}_{S_+^{n-1}}\times\Psi_{x_0})\circ (\mathrm{id}_{S_+^{n-1}}\times\oo{\Phi}_{x_0})^{-1},$$so the last condition is equivalent to $\Lambda(x,(\Psi_{x_0}\circ\oo{\Phi}^{-1}_{x_0})(q))\in Z(h)$. We have thus shown that $$\chi(Z(h)\cap\bot^{(1)})[x]=\{ q\in Q:\,\Lambda(x,(\Psi_{x_0}\circ\oo{\Phi}_{x_0}^{-1})(q))\in Z(h)\} .$$ Since the map $\Psi_{x_0}\circ\oo{\Phi}_{x_0}^{-1}\colon Q\to U$ is a~diffeomorphism, axiom (H$_1$) and Lemma \ref{111} imply that:\begin{equation*}
\begin{split}
\chi(Z(h)&\cap\bot^{(1)})[x]\not\in\I_Q\\
&\Longleftrightarrow\{(\lambda ,y)\in U :\,\Lambda(x,\lambda ,y)\in Z(h)\}\not\in\I_n\\
&\Longleftrightarrow\{(\lambda ,y)\in U :\,\oo{\Phi}_x(\lambda ,\w{\psi}_x^{-1}(y))\in Z(h)\}\not\in\I_n\\
&\Longleftrightarrow\{(\lambda ,y)\in\R^\ast\times (\R^{n-1})^\ast :\,\oo{\Phi}_x(\lambda ,\w{\psi}_x^{-1}(y))\in Z(h)\}\not\in\I_n\\
&\Longleftrightarrow\{(\lambda ,y)\in\R^\ast\times P_x^\ast :\,\oo{\Phi}_x(\lambda ,y)\in Z(h)\}\not\in\I_{\R^\ast\times P_x^\ast}\\
&\Longleftrightarrow\{(\lambda ,y)\in (\R^\ast\times P_x^\ast )\setminus T_x :\,\Phi_x(\lambda ,y)\in Z(h)\}\not\in\I_{(\R^\ast\times P_x^\ast)\setminus T_x}\\
&\Longleftrightarrow Z(h)\cap Q(x)\not\in\I_{Q(x)}.
\end{split}
\end{equation*}
Thus \eqref{QS3} gives
\begin{equation*}
\{ x\in P:\, Z(h)\cap Q(x)\not\in\I_{Q(x)}\}\in\I_{S_+^{n-1}}.
\end{equation*}
Since $S_+^{n-1}\setminus P\in\I_{S_+^{n-1}}$, we have also
\begin{equation}\label{S1}
Z(h)\cap Q(x)\in\I_{Q(x)}\quad\I_{S_+^{n-1}}\mbox{-(a.e.)}.
\end{equation}

For any $x\in S_+^{n-1}$ define $\Gamma_x\colon\R^\ast\times P_x^\ast\to\bot^\ast$ and $\Theta_x\colon\R^\ast\times P_x^\ast\to\bot^\ast$ as
\begin{equation*}
\Gamma_x(\lambda ,y)=\Bigl({\|y\|^2\over\lambda}x,-y\Bigr)\,\,\mbox{ and }\,\,\Theta_x(\lambda ,y)=(\lambda x,y),
\end{equation*}
and put $R(x)=\Gamma_x(\R^\ast\times P_x^\ast)$, $S(x)=\Theta_x(\R^\ast\times P_x^\ast)$. An argument similar to the one above shows that $R(x)$, for $x\in S_+^{n-1}$, are submanifolds of $\bot^\ast$, $\mathcal{C}^\infty$-diffeomorphic each to others, and the same is true for $S(x)$'s. Moreover, the set $$\bot^{(2)}:=\{ (t,u)\in\bot^\ast :\, t_n\not=0\mbox{ and }u\not=0\}=\bigcup_{x\in S_+^{n-1}}R(x)=\bigcup_{x\in S_+^{n-1}}S(x)$$is $\mathcal{C}^\infty$-diffeomorphic to $S_+^{n-1}\times R$ and $S_+^{n-1}\times S$, where $R$ and $S$ are \lq\lq model\rq\rq\ manifolds for all $R(x)$'s and for all $S(x)$'s, respectively. Arguing further, analogously as above, we also infer that
\begin{equation}\label{S2}
Z(h)\cap R(x)\in\I_{R(x)}\,\,\mbox{ and }\,\, Z(h)\cap S(x)\in\I_{S(x)}\quad\I_{S_+^{n-1}}\mbox{-(a.e.)}.
\end{equation}

According to \eqref{S1} and \eqref{S2} there is a set $S_0\in\I_{S_+^{n-1}}$ with
\begin{equation}\label{S0}
\left\{\begin{array}{l}
Z(h)\cap Q(x)\in\I_{Q(x)},\\
Z(h)\cap R(x)\in\I_{R(x)},\\
Z(h)\cap S(x)\in\I_{S(x)}\end{array}\right.
\end{equation}
for $x\in S_+^{n-1}\setminus S_0$.

At the moment, assume that $n=2$. Applying Lemma \ref{sfera} to the set $$A:=S_0\cup (-S_0)\cup\{(-1,0),(1,0)\}\in\I_{S^1},$$and changing signs of vectors of the obtained basis as required, we get an orthogonal basis $(x^{(1)},x^{(2)})$ of $\R^2$ whose each element $x$ satisfies conditions \eqref{S0}.

Now, we shall prove that for each $i\in\{ 1,2\}$ the function $h_i\colon\R\to G$ given by $h_i(\lambda)=h(\lambda x^{(i)})$ satisfies
\begin{equation}\label{hi}
h_i(\lambda+\mu)=h_i(\lambda)+h_i(\mu)\quad\Omega(\I_{(0,\infty)})\mbox{-(a.e.)},
\end{equation}
where $\Om(\I_{(0,\infty)})=\{ A\subset (0,\infty)^2:\, A[x]\in\I_{(0,\infty)}\,\,\,\I_{(0,\infty)}\mbox{-(a.e.)}\}$ is the so called conjugate ideal. Plainly, condition \eqref{hi} would imply that the same is true with $(0,\infty)$ replaced by $(-\infty,0)$, due to the oddness of the function $h$.

Fix $i\in\{ 1,2\}$. In view of \eqref{S0}, with $x$ replaced by $x^{(i)}$, there is a set $C_i\in\I_{\R^\ast\times P_{x^{(i)}}^\ast}$ such that
\begin{equation}\label{Cii}
\left\{\begin{array}{l}
\Bigl(\lambda x^{(i)}+y,\displaystyle{{\|y\|^2\over\lambda}x^{(i)}}-y\Bigr)\in\bot^\ast\setminus  Z(h),\\
\Bigl(\displaystyle{{\|y\|^2\over\lambda}}x^{(i)},-y\Bigr)\in\bot^\ast\setminus  Z(h),\vspace*{1mm}\\
\,\bigl(\lambda x^{(i)},y\bigr)\in\bot^\ast\setminus Z(h)
\end{array}\right.
\end{equation}
for $(\lambda ,y)\in (\R^\ast\times P_{x^{(i)}}^\ast)\setminus C_i$ (note that $T_{x^{(i)}}\in\I_{\R^\ast\times P_{x^{(i)}}^\ast}$, so we may include the set $T_{x^{(i)}}$ into $C_i$ and we see that the difference between the domain of $\Phi_{x^{(i)}}$ and the domains of $\Gamma_{x^{(i)}}$, $\Theta_{x^{(i)}}$ causes no trouble at all). Therefore, for all $\lambda\in\R$ except a set $\Lambda_i\in\I_1$ the conjunction \eqref{Cii} holds true for all $y\in P_{x^{(i)}}\setminus Y_i(\lambda)$ with $Y_i(\lambda)\in\I_{P_{x^{(i)}}}$. Let $$B_i(\lambda)=\left\{{\|y\|^2\over\lambda}:\, y\in P_{x^{(i)}}\setminus Y_i(\lambda)\right\} .$$Then, obviously, $\R\setminus B_i(\lambda)\in\I_{(0,\infty)}$ for each positive $\lambda\not\in\Lambda_i$, whereas $\R\setminus B_i(\lambda)\in\I_{(-\infty,0)}$ for each negative $\lambda\not\in\Lambda_i$. For every pair $(\lambda ,\mu)$ with $\lambda\not\in\Lambda_i$ and $\mu\in B_i(\lambda)$, $\mu={\|y\|^2\over\lambda}$, we have
\begin{equation*}
\begin{split}
h_i(\lambda+\mu)&=h\Bigl(\lambda x^{(i)}+y+{\|y\|^2\over\lambda}x^{(i)}-y\Bigr)=h(\lambda x^{(i)}+y)+h\Bigl({\|y\|^2\over\lambda}x^{(i)}-y\Bigr)\\
&=h(\lambda x^{(i)})+h(y)+h\Bigl({\|y\|^2\over\lambda}x^{(i)}\Bigr)+h(-y)=h_i(\lambda)+h_i(\mu),
\end{split}
\end{equation*}
which proves \eqref{hi}. Applying the theorem of de Bruijn \cite{bruijn} separately to the functions $h_i\vert_{(0,\infty)}$ and $h_i\vert_{(-\infty,0)}$ we get two additive mappings $b_i^\prime\colon (0,\infty)\to G$ and $b_i^{\prime\prime}\colon (-\infty,0)\to G$ which coincide with these two restrictions of $h_i$ almost everywhere in $(0,\infty)$ and $(-\infty,0)$, respectively. However, since $h$ is odd, the extensions of both $b_i^\prime$ and $b_i^{\prime\prime}$ to the whole real line have to be the same. As a result, there is an additive function $b_i\colon\R\to G$ such that $h_i(\lambda)=b_i(\lambda)$ for $\lambda\in\R\setminus Z_i$ with a certain $Z_i\in\I_1$.

Define a function $b\colon\R^2\to G$ by $b(x)=b_1(\lambda_1)+b_2(\lambda_2)$, where $\lambda_i$ is the $i$th coordinate of $x$ with respect to the basis $(x^{(1)},x^{(2)})$. Plainly, $b$ is an additive function. It remains to show that $h(x)=b(x)$ $\I_2$-(a.e.).

Recall that for every $x\in X=\R\times\R^\ast$ the mapping $\Psi_x$ defined by \eqref{psipsi} yields a $\mathcal{C}^\infty$-diffeomorphism between $\R^\ast\times P_x^\ast$ and $\R^\ast\times\R^\ast$. In particular, we have $C:=\Psi_{x^{(1)}}(C_1)\in\I_2$ and
\begin{equation}\label{Ciii}
\bigl(\lambda x^{(1)},\w{\psi}_{x^{(1)}}^{-1}(y)\bigr)\in\bot^\ast\setminus Z(h)\quad\mbox{for }(\lambda ,y)\in\R^2\setminus C.
\end{equation}
Define $\Delta\colon\R^2\to\R^2$ by $$\Delta(\lambda_1,\lambda_2)=\bigl(\lambda_1,\w{\psi}_{x^{(1)}}(\lambda_2x^{(2)})\bigr).$$
Plainly, $\Delta$ is a $\mathcal{C}^\infty$-diffeomorphism, so $\Delta^{-1}(C)\in\I_2$. Therefore, $$\Delta^{-1}(C)\cup(Z_1\times\R)\cup(\R\times Z_2)\in\I_2$$and for each pair $(\lambda_1,\lambda_2)\in\R^2$ outside this set condition \eqref{Ciii} implies $(\lambda_1x^{(1)},\lambda_2x^{(2)})\in\bot^\ast\setminus Z(h)$, thus
\begin{equation*}
\begin{split}
h(\lambda_1x^{(1)}+\lambda_2x^{(2)})&=h(\lambda_1x^{(1)})+h(\lambda_1x^{(1)})=h_1(\lambda_1)+h_2(\lambda_2)\\
&=b_1(\lambda_1)+b_2(\lambda_2)=b(\lambda_1x^{(1)}+\lambda_2x^{(2)}).
\end{split}
\end{equation*}
By the isomorphism, which to every $x\in\R^2$ assigns its coordinates in the basis $(x^{(1)},x^{(2)})$, we have $h(x)=b(x)$ $\I_2$-(a.e.) and our assertion for $n=2$ follows.

In the sequel, assume that $n\geq 3$ and the assertion holds true for $n-1$ in the place of $n$.

Define $\Or(n-1,n)^\prime$ to be the set of all $(n-1)$-tuples from $\Or(n-1,n)$ generating a subspace of $\R^n$ whose orthogonal complement is spanned by a vector  $(x_1,\ldots ,x_n)$ with $x_n\not=0$. In other words, $$\Or(n-1,n)^\prime=\left\{ (x^{(1)},\ldots ,x^{(n-1)})\in\Or(n-1,n):\,\pm{x^{(1)}\wedge\ldots\wedge x^{(n-1)}\over\|x^{(1)}\wedge\ldots\wedge x^{(n-1)}\|}\in S_+^{n-1}\right\},$$
where $\wedge$ stands for the wedge product in $\R^n$. This set, being an open subset of $\Or(n-1,n)$, is its submanifold having the same dimension. Consider the mapping $\Om\colon S_+^{n-1}\times\Or(n-1,n-1)\to\Or(n-1,n)^\prime$ defined by $$\Om(x,x^{(1)},\ldots ,x^{(n-1)})=\bigl(\w{\psi}_x^{-1}(x^{(1)}),\ldots ,\w{\psi}_x^{-1}(x^{(n-1)})\bigr).$$The values of $\Om$ indeed belong to $\Or(n-1,n)^\prime$, since for each $x\in X$ the function $\psi_x$ is an isometry, being a linear map determined by the orthogonal matrix $\boldsymbol{Y}(x)^{-1}$. Furthermore, $\Om$ is bijective with the inverse $\Om^{-1}$ given by $$\Om^{-1}(y^{(1)},\ldots ,y^{(n-1)})=\bigl(x,\w{\psi}_x(y^{(1)}),\ldots ,\w{\psi}_x(y^{(n-1)})\bigr),$$where $$x=\pm{y^{(1)}\wedge\ldots\wedge y^{(n-1)}\over\|y^{(1)}\wedge\ldots\wedge y^{(n-1)}\|}$$and the sign depends on which of the two components of $\Or(n-1,n)^\prime$ contains $(y^{(1)},\ldots ,y^{(n-1)})$. By the above formulas, $\Om$ is a $\mathcal{C}^\infty$-diffeomorphism.

Put $$Z=\{ (y^{(1)},\ldots ,y^{(n-1)})\in\Or(n-1,n)^\prime :\, (y^{(1)},y^{(2)})\in Z(h)\} .$$Then Lemma \ref{submersion2} implies $Z\in\I_{\Or(n-1,n)^\prime}$, since $Z(h)\in\I_\bot$ (i.e. $Z(h)\in\I_{\Or(2,n)}$) is the image of $Z$ through the $\mathcal{C}^\infty$-submersion $(y^{(1)},\ldots ,y^{(n-1)})\mapsto (y^{(1)},y^{(2)})$. Therefore, we have $\Om^{-1}(Z)\in\I_{S_+^{n-1}\times\Or(n-1,n-1)}$, hence $\Om^{-1}(Z)[x]\in\I_{\Or(n-1,n-1)}$ is valid $\I_{S_+^{n-1}}$-(a.e.), which translates into the fact that the set
\begin{equation*}
A(x):=\bigl\{(x^{(1)},\ldots ,x^{(n-1)})\in\Or(n-1,n-1):\,\bigl(\w{\psi}_x^{-1}(x^{(1)}),\w{\psi}_x^{-1}(x^{(1)})\bigr)\in Z(h)\bigr\}
\end{equation*}
belongs to $\I_{\Or(n-1,n-1)}$ for every $x\in S_+^{n-1}$ except a set from $\I_{S_+^{n-1}}$. By virtue of Lemma \ref{O(k,m)}, for each such $x$ we must have
\begin{equation}\label{OZ2}
\bigl\{(x^{(1)},x^{(2)})\in\Or(2,n-1):\,\bigl(\w{\psi}_x^{-1}(x^{(1)}),\w{\psi}_x^{-1}(x^{(2)})\bigr)\in Z(h)\bigr\}\in\I_{\Or(2,n-1)}.
\end{equation}
Hence, putting $\bot_x=\{(t,u)\in P_x\times P_x:\, (t,u)\in\bot\}$ we infer that the condition
\begin{equation}\label{OZ3}
h(t+u)=h(t)+h(u)\quad\I_{\bot_x^\ast}\mbox{-(a.e.)}
\end{equation}
is valid $\I_{S_+^{n-1}}$-(a.e.). Consequently, we may pick a particular $x\in S_+^{n-1}$ satisfying both \eqref{S0} and \eqref{OZ3}. By virtue of our inductive hypothesis and some isometry formalities (identifying $P_x$ with $\R^{n-1}$), condition \eqref{OZ3} yields the existence of an additive function $b_x\colon P_x\to G$ such that $h(t)=b_x(t)$ for $t\in P_x\setminus Y$ with a certain $Y\in\I_{P_x}$. Moreover, by an earlier argument, there is also an additive function $b_1\colon\R\to G$ such that $h(\lambda x)=b_1(\lambda)$ for $\lambda\in\R\setminus Z_1$ with a certain $Z_1\in\I_1$. Finally, there is a set $C_1\in\I_{\R\times P_x}$ with $(\lambda x,y)\in\bot^\ast\setminus Z(h)$ whenever $(\lambda ,y)\in (\R\times P_x)\setminus C_1$.

Define a function $b\colon\R^n\to G$ by the formula $b(\lambda x+y)=b_1(\lambda)+b_x(y)$ for $\lambda\in\R$ and $y\in P_x$. Then $b$ is additive and for each pair $(\lambda ,y)\in\R\times P_x$ outside the set $$C_1\cup (Z_1\times P_x)\cup (\R\times Y)\in\I_{\R\times P_x}$$we have $$h(\lambda x+y)=h(\lambda x)+h(y)=b_1(\lambda)+b_x(y)=b(\lambda x+y),$$which completes the proof.
\end{proof}

\begin{lemma}\label{L2}
If a function $h\colon\R^n\to G$ satisfies $h(x)=h(-x)$ $\I_n${\rm -(a.e.)\/} and $h(x+y)=h(x)+h(y)$ $\I_\bot${\rm -(a.e.)\/}, then there is an additive function $a\colon\R\to G$ such that $h(x)=a\left(\| x\|^2\right)$ $\I_n${\rm -(a.e.)\/}.
\end{lemma}
\begin{proof}
For any $r\geq 0$ let $S^{n-1}(r)=\{ x\in\R^n:\,\|x\|=r\}$. By the natural identification, we have $(\R^n)^\ast\sim (0,\infty)\times S^{n-1}$. Therefore, for every $A\in\I_n$ there is a set $R(A)\in\I_{(0,\infty)}$ such that $A\cap S^{n-1}(r)\in\I_{S^{n-1}(r)}$ for $r\in (0,\infty)\setminus R(A)$. In the first part of the proof we will show the following claim: there exists a set $A\in\I_n$ such that for each $r\in (0,\infty)\setminus R(A)$ the function $h$ is constant $\I_{S^{n-1}(r)}$-(a.e.) on $S^{n-1}(r)$, more precisely -- that $h|_{S^{n-1}(r)}$ is constant outside the set $A\cap S^{n-1}(r)$.

We start with the following observation: there is $T\in\I_\bot$ such that $h(t+u)=h(u-t)$ whenever $(t,u)\in\bot^\ast\setminus T$. Let $E=\{ x\in\R^n:\, h(x)=h(-x)\}$ and $H=(-D(h))\cap D(h)\cap E$; then $\R^n\setminus H\in\I_n$. Define
\begin{equation}\label{Tdef}
T=\{ (t,u)\in\bot^\ast:\, t\not\in H\}\cup\{ (t,u)\in\bot^\ast:\, t\in H\mbox{ and }u\not\in E_t(h)\cap E_{-t}(h)\} .
\end{equation}
Then for every $(t,u)\in\bot^\ast\setminus T$ we have $h(t+u)=h(t)+h(u)$ and $h(u-t)=h(u)+h(-t)$. Moreover, we have also $h(t)=h(-t)$, hence $h(t+u)=h(u-t)$, as desired. In order to show that $T\in\I_\bot$ note that it is equivalent to $T\cap\bot^\prime\in\I_{\bot^\prime}$, where $\bot^\prime$ may be identified with $X\times\R^{n-1}$. The first summand in \eqref{Tdef}, after intersecting with $\bot^\prime$, is then identified with $(X\setminus H)\times\R^{n-1}\in\I_{2n-1}$, whereas for each pair $(t,u)$ from the second summand we have either $(t,u)\in Z(h)$, or $(-t,u)\in Z(h)$, which shows that it belongs to $\I_\bot$. Consequently, $T\in\I_\bot$.

Define $\Phi\colon\bot^\ast\to\R^n\times\R^n$ by putting $\Phi(t,u)=(t+u,u-t)$. It is evident that $\Phi$ is a $\mathcal{C}^\infty$-immersion and yields a homeomorphism between $\bot^\ast$ and $$M:=\Phi(\bot^\ast)=\bigcup_{r\in (0,\infty )}(S^{n-1}(r)\times S^{n-1}(r)).$$Therefore, \cite[Theorem 11.17]{tu} implies that $M$ is a manifold. Moreover, $\Phi\colon\bot^\ast\to M$ is a $\mathcal{C}^\infty$-diffeomorphism, thus $\Phi(T)\in\I_M$. Since the mapping $(x,y)\mapsto (x,y/\|x\|)$ yields $M\sim (\R^n)^\ast\times S^{n-1}$, there exists a set $A\in\I_n$ such that for every $x\in\R^n\setminus A$ we have $$(x,y)\not\in\Phi (T)\quad\I_{S^{n-1}(\| x\| )}\mbox{-(a.e.)}.$$By the property of the set $T$, $(x,y)\not\in\Phi (T)$ implies $h(x)=h(y)$. Now, for any $r\in (0,\infty )\setminus R(A)$ and for arbitrary $x,y\in\R^n\setminus A$ with $\| x\| =\| y\| =r$, we have $$(x,z), (y,z)\not\in\Phi (T)\quad\I_{S^{n-1}(r)}\mbox{-(a.e.)},$$hence $h(x)=h(z)=h(y)$, which completes the proof of our claim.

There is a function $g\colon\R^n\to G$ which is constant on every sphere $S^{n-1}(r)$ and such that $h(x)=g(x)$ for $x\in\R^n\setminus A$. Therefore, there is also a function $\p\colon [0,\infty)\to G$ satisfying $g(x)=\p\left(\|x\|^2\right)$ for every $x\in\R^n$. We are going to show that
\begin{equation}\label{phi}
\p(\lambda+\mu)=\p(\lambda)+\p(\mu)\quad\Om(\I_{(0,\infty)})\mbox{-(a.e.)}.
\end{equation}

Put $$B=\{ (x,y)\in\bot^\ast :\mbox{ either }x\in A\mbox{, or }y\in A\mbox{, or }x+y\in A\}$$and observe that $B\in\I_\bot$, whence also $Z:=Z(h)\cup B\in\I_\bot$. Let $$D=\{ x\in (\R^n)^\ast :\, (x,y)\not\in Z\,\,\,\,\I_{P_x}\mbox{-(a.e.)}\} .$$By an argument similar to the one applied to $D(h)$, we infer that $X\setminus D\in\I_X$, hence $\R^n\setminus D\in\I_n$. For each $x\in\R^n$ put $E_x=\{ y\in P_x:\, (x,y)\not\in Z\}$; then $P_x\setminus E_x\in\I_{P_x}$ provided $x\in D$. Let also $D^\prime=\{\|x\|^2:\, x\in D\}$; then $(0,\infty)\setminus D^\prime\in\I_{(0,\infty)}$.

Fix arbitrarily $\lambda\in D^\prime$ and choose any $x\in D$ satisfying $\sqrt{\lambda}=\|x\|$. Put $E(\lambda)=\{\|y\|^2:\, y\in E_x\}$ (then $(0,\infty)\setminus E(\lambda)\in\I_{(0,\infty)}$) and pick any $\mu\in E(\lambda)$. Then $\sqrt{\mu}=\|y\|$ for some $y\in E_x$, which implies $(x,y)\not\in Z$. Applying the facts that $x+y\not\in A$, $(x,y)\not\in Z(h)$, $x\not\in A$ and $y\not\in A$, consecutively, we obtain
\begin{equation*}
\begin{split}
\p(\lambda+\mu)&=g(x+y)=h(x+y)\\
&=h(x)+h(y)=g(x)+g(y)=\p(\lambda)+\p(\mu),
\end{split}
\end{equation*}
which proves \eqref{phi}.

By the theorem of de Bruijn, there is an additive function $a\colon\R\to G$ such that $\p(\lambda)=a(\lambda)$ for $\lambda\in [0,\infty)\setminus Y$ with $Y\in\I_{[0,\infty)}$. Then the equality $h(x)=a\left(\|x\|^2\right)$ holds true for $x\in\R^n\setminus (A\cup C)$, where $C=\{ x\in\R^n:\,\|x\|^2\in Y\}\in\I_n$. Thus, the proof has been completed.
\end{proof}

To finish the proof of our Theorem we shall combine Lemmas \ref{L0}, \ref{L1} and \ref{L2} to get additive functions $a\colon\R\to G$ and $b\colon\R^n\to G$ such that $$2\left(f(x)-a(\|x\|^2)-b(x)\right)=0\quad\I_n\mbox{-(a.e.)}.$$The only thing left to be proved is the following fact in the spirit of \cite[Lemma 2]{baron_ratz}.
\begin{lemma}
If a function $h\colon\R^n\to G$ satisfies $2h(x)=0$ $\I_n${\rm -(a.e.)\/} and $h(x+y)=h(x)+h(y)$ $\I_\bot${\rm -(a.e.)\/}, then $h(x)=0$ $\I_n${\rm -(a.e.)\/}.
\end{lemma}
\begin{proof}
For every $x\in\R^n$ put $g(x)=h(x)-h(-x)$. Applying Lemmas \ref{L0} and \ref{L1} we get an additive function $b\colon\R^n\to G$ such that $g(x)=b(x)$ $\I_n$-(a.e.). Therefore $$g(x)=2b\left({x\over 2}\right)=2h\left({x\over 2}\right)-2h\left(-{x\over 2}\right)=0\quad\I_n\mbox{-(a.e.)},$$i.e. $h(x)=h(-x)$ $\I_n$-(a.e.). Now, by virtue of Lemma \ref{L2}, there is an additive function $a\colon\R\to G$ satisfying $h(x)=a(\|x\|^2)$ $\I_n$-(a.e.). Consequently, $$h(x)=a\left(2\left\|{1\over\sqrt{2}}x\right\|^2\right)=2a\left(\left\|{1\over\sqrt{2}}x\right\|^2\right)=2h\left({1\over\sqrt{2}}x\right)=0\quad\I_n\mbox{-(a.e.)}.$$\qedhere
\end{proof}

\bibliographystyle{amsplain}

\end{document}